\documentclass[11pt]{amsart}
\usepackage{amssymb}
\usepackage{amsmath}

\newtheorem{theorem}{Theorem}[section]
\newtheorem{lemma}[theorem]{Lemma}

\newtheorem{corollary}[theorem]{Corollary}
\theoremstyle{definition}
\newtheorem{definition}[theorem]{Definition}

\theoremstyle{remark}
\newtheorem{remark}[theorem]{Remark}
\numberwithin{equation}{section}

\def\id{\mathop{\rm id}}

\def\id{{\bf 1}\!\!{\rm I}}

\def\bn{{\mathbb N}}

\def\br{{\mathbb R}}

\def\l{\lambda}

\def\i{\varepsilon}

\def\N{\mathbb{N}}

\def\a{\alpha}

\def\b{\beta}

\def\d{\delta}

\def\cf{{\mathcal F}}

\def\cf{\mathcal{ F}}

\def\ck{\mathcal{K}}

\def\gu{{\frak U}}

\begin{document}
\setcounter{page}{1}

\title[Uniform ergodicity and perturbation bounds]{Uniform ergodicities and perturbation bounds of Markov chains on ordered Banach spaces}

\author[Nazife Erkur\c{s}un \"Ozcan, Farrukh Mukhamedov]{Nazife Erkur\c{s}un \"Ozcan$^1$ and Farrukh Mukhamedov$^2$$^*$}

\address{$^{1}$ Department of Mathematics, Faculty of Science, Hacettepe University, Ankara, 06800,Turkey.}
\email{{erkursun.ozcan@hacettepe.edu.tr}}

\address{$^{2}$ Department of Computational and Theoretical Science, Faculty of Science, International Islamic University Malaysia, Kuantan, Pahang, 25710,
Malaysia.} \email{{far75m@yandex.ru; farrukh\_m@iium.edu.my}}

%\dedicatory{This paper is dedicated to Professor ABCD}



\date{Received: xxxxxx; Revised: yyyyyy; Accepted: zzzzzz.
\newline \indent $^{*}$ Corresponding author}

\begin{abstract}
It is known that Dobrushin's ergodicity coefficient is one of the
effective tools in the investigations of limiting behavior of
Markov processes.  Several interesting properties of the
ergodicity coefficient of a positive mapping defined on ordered
Banach space with a base have been studied. In this paper, we
consider uniformly mean ergodic and asymptotically stable Markov
operators on ordered Banach spaces. In terms of the ergodicity
coefficient, we prove uniform mean ergodicity criterion in terms
of the ergodicity coefficient. Moreover, we develop the
perturbation theory for uniformly asymptotically stable Markov
chains on ordered Banach spaces. In particularly, main results
open new perspectives in the perturbation theory for quantum
Markov processes defined on von Neumann algebras. Moreover, by
varying the Banach spaces one can obtain several interesting
results in both classical and quantum settings as well.
\end{abstract} \maketitle

\section{Introduction}

It is well-known that the transition probabilities $P(x,A)$
(defined on a measurable space $(E,\cf)$) of Markov processes
naturally define a linear operator by $ Tf(x)=\int f(y)P(x,dy)$,
which is called {\it Markov operator} and acts on $L^1$-spaces.
The study of the entire process can be reduced to the study of the
limit behavior of the corresponding Markov operator (see
\cite{K}). When we look at quantum analogous of Markov processes,
which naturally appear in various directions of quantum physics
such as quantum statistical physics and quantum optics etc. In
these studies it is important to elaborate with associated quantum
dynamical systems (time evolutions of the system) \cite{RKW},
which eventually converge to a set of stationary states. From the
mathematical point of view, ergodic properties of quantum  Markov
operators were investigated by many authors. We refer a reader to
\cite{AH,FR1,NSZ} for further details relative to some differences
between the classical and the quantum situations.

In \cite{SZ} it was proposed to investigate ergodic properties of
Markov operator on abstract framework, i.e. on ordered Banach
spaces. Since the study of several properties of physical and
probabilistic processes in abstract framework is convenient and
important (see \cite{Alf}). Some applications of this scheme in
quantum information have been discussed in \cite{RKW}. We
emphasize that the classical and quantum cases confine to this
scheme. We point out that in this abstract scheme one considers an
ordered normed spaces and mappings of these spaces (see
\cite{Alf}). Moreover, in this setting mostly, certain ergodic
properties of Markov operators s were considered and investigated
in \cite{B,EW1,RKW}. Nevertheless, the question about the
sensitivity of stationary states and perturbations of the Markov
chain are not explored well. Very recently in \cite{SW},
perturbation bounds have been found for a quantum Markov chains
acting on finite dimensional algebras.

On the other hand, it is known \cite{Kar,Mit} that Dobrushin's
ergodicity coefficient is one of the effective tools in the
investigations of limiting behavior of Markov processes (see
\cite{IS,Se2} for review). In \cite{M0,M01} we have defined such
an ergodicity coefficient $\d(T)$ of a positive mapping $T$
defined on ordered Banach space with a base, and studied its
properties. In this paper, we consider uniformly mean ergodic and
uniformly asymptotical stable Markov operators on ordered Banach
spaces. In terms of the ergodicity coefficient, we prove the
equivalence of uniform and weak mean ergodicities of Markov
operators. This result allowed us to establish a category theorem
for uniformly mean ergodic Markov operators. Furthermore,
following some ideas of \cite{Kar,Mit} and using properties of
$\d(T)$, we develop the perturbation theory for uniformly
asymptotical stable Markov chains in the abstract scheme. Our
results open new perspectives in the perturbation theory for
quantum Markov processes in more general von Neumann algebras
setting, which have significant applications in quantum theory
\cite{RKW}.

\section{Preliminaries}

In this section we recall some necessary definitions and fact
about ordered Banach spaces.

 Let $X$ be an ordered vector space
with a cone $X_+=\{x\in X: \ x\geq 0\}$. A subset $\ck$ is called
a {\it base} for $X$, if one has $\ck=\{x\in X_+:\ f(x)=1\}$ for
some strictly positive (i.e. $f(x)>0$ for $x>0$) linear functional
$f$ on $X$. An ordered vector space $X$ with generating cone $X_+$
(i.e. $X=X_+-X_+$) and a fixed base $\ck$, defined by a functional
$f$, is called {\it an ordered vector space with a base}
\cite{Alf}. In what follows, we denote it as $(X,X_+,\ck,f)$. Let
$U$ be the convex hull of the set $\ck\cup(-\ck)$, and let
$$
\|x\|_{\ck}=\inf\{\l\in\br_+:\ x\in\l U\}.
$$
Then one can see that $\|\cdot\|_{\ck}$ is a seminorm on $X$.
Moreover, one has $\ck=\{x\in X_+: \ \|x\|_{\ck}=1\}$,
$f(x)=\|x\|_{\ck}$ for $x\in X_+$. If the set $U$ is linearly
bounded (i.e. for any line $\ell$ the intersection $\ell\cap U$ is
a bounded set), then $\|\cdot\|_{\ck}$ is a norm, and in this case
$(X,X_+,\ck,f)$ is called {\it an ordered normed space with a
base}. When  $X$ is complete with respect to the norm
$\|\cdot\|_{\ck}$ and the cone $X_+$ is closed, then
$(X,X_+,\ck,f)$ is called {\it an ordered Banach space with a base
(OBSB)}. In the sequel, for the sake of simplicity instead of
$\|\cdot\|_{\ck}$ we will use usual notation $\|\cdot\|$.

Let us provide some examples of OBSB.

\begin{itemize}
\item[1.] Let $M$ be a von Neumann algebra. Let $M_{h,*}$ be the
Hermitian part of the predual space $M_*$ of $M$. As a base $\ck$
we define the set of normal states of $M$. Then
$(M_{h,*},M_{*,+},\ck,\id)$ is a OBSB, where $M_{*,+}$ is the set
of all positive functionals taken from $M_*$, and $\id$ is the
unit in $M$.

\item[2.] Let $X=\ell_p$, $1<p<\infty$. Define
$$
X_+=\bigg\{\mathbf{x}=(x_0,x_1,\dots,x_n,\dots)\in\ell_p: \
x_0\geq \bigg(\sum_{i=1}^\infty|x_i|^p\bigg)^{1/p}\bigg\}
$$
and $f_0(\mathbf{x})=x_0$. Then $f_0$ is a strictly positive
linear functional. In this case, we define $\ck=\{x\in X_+: \
f_0(\mathbf{x})=1\}$. Then one can see that $(X,X_+,\ck,f_0)$ is a
OBSB. Note that the norm $\|\cdot\|_{\ck}$ is equivalent to the
usual $\ell_p$-norm.
\end{itemize}

Let $(X,X_+,\ck,f)$ be an OBSB. It is well-known (see
\cite[Proposition II.1.14]{Alf}) that every element $x$ of OBSB
admits a decomposition $x=y-z$, where $y,z\geq 0$ and
$\|x\|=\|y\|+\|z\|$. From this decomposition, we obtain the
following fact.

\begin{lemma}\label{3.2}\cite{M0} For every $x,y\in X$ such that $x-y\in N$ there
exist $u,v\in \ck$ with
$$
x-y=\frac{\|x-y\|}{2}(u-v).
$$
\end{lemma}

Let $(X,X_+,\ck,f)$ be an OBSB. A linear operator $T:X\to X$ is
called positive, if $Tx\geq 0$ whenever $x\geq 0$. A positive
linear operator $T:X\to X$ is called {\it Markov}, if
$T(\ck)\subset\ck$. It is clear that $\|T\|=1$, and its adjoint
mapping $T^*: X^*\to X^*$ acts in ordered Banach space $X^*$ with
unit $f$, and moreover, one has $T^*f=f$. Note that in case of
$X=\br^n$, $X_+=\br_+^n$ and $\ck=\{(x_i)\in\br^n: \ x_i\geq 0, \
\sum_{i=1}^n x_i=1\}$, then for any Markov operator $T$ acting on
$\br^n$, the conjugate operator $T^*$ can be identified with a
usual stochastic matrix. Now for each $y\in X$ we define a linear
operator $T_y: X\to X$ by $T_y(x)=f(x)y$. For a given operator $T$
we denote
$$
A_n(T)=\frac{1}{n}\sum_{k=0}^{n-1}T^{k}, \ \ n\in\bn.
$$

\begin{definition} A Markov operator $T: X\to X$ is called
\begin{enumerate}
\item[(i)] {\it uniformly asymptotically stable} if there exist an
element $y_0\in\ck$ such that
$$
\lim_{n\to\infty}\|T^{n}-T_{y_0}\|=0;
$$
\item[(ii)] {\it uniformly mean ergodic} if there exist an element
$y_0\in \ck$ such that
$$
\lim_{n\to\infty} \bigg\|A_n(T)-T_{y_0}\bigg\|=0;
$$
\item[(iii)] \textit{weakly ergodic} if one has
$$
\lim_{n\to\infty}\sup_{x,y\in \ck}\|T^{n}x-T^{n}y\|=0;
$$
\item[(iv)] {\it weakly mean ergodic} if one has
$$
\lim_{n\to\infty} \sup_{x,y\in \ck} \|A_n(T)x - A_n(T)y\|=0.
$$
\end{enumerate}
\end{definition}

\begin{remark}\label{dd0} We notice that uniform asymptotical
stability implies uniform mean ergodicity. Moreover, if $T$ is
uniform mean ergodic, then $y_0$, corresponding to $T_{y_0}$, is a
fixed point of $T$. Indeed, taking limit in the equality
$$
\bigg(1+\frac{1}{n}\bigg)A_{n+1}(T)-\frac{1}{n}I=TA_{n}(T)
$$
we find $TT_{y_0}=T_{y_0}$, which yields $Ty_0=y_0$. We stress
that every uniformly mean ergodic Markov operator has a unique
fixed point.
\end{remark}

Let $(X,X_+,\ck,f)$ be an OBSB  and $T:X\to X$ be a linear bounded
operator. Letting
\begin{equation}
\label{NN} N=\{x\in X: \ f(x)=0\},
\end{equation}
we define
\begin{equation}
\label{db} \d(T)=\sup_{x\in N,\ x\neq 0}\frac{\|Tx\|}{\|x\|}.
\end{equation}

The quantity $\d(T)$ is called the \textit{Dobrushin's ergodicity
coefficient} of $T$ (see \cite{M0}).

\begin{remark} We note that if $X^*$ is a commutative algebra, the notion of
the Dobrushin's ergodicity coefficient was studied in
\cite{C},\cite{D} (see \cite{IS,Se2} for review). In a
non-commutative setting, i.e. when $X^*$ is a von Neumann algebra,
such a notion was introduced in \cite{M}. We should stress that
such a coefficient has been independently defined in \cite{GQ}.
Furthermore, for particular cases, i.e. in a non-commutative
setting, the coefficient explicitly has been calculated for
quantum channels (i.e. completely positive maps).
\end{remark}

The next result establishes several properties of the Dobrushin's
ergodicity coefficient.

\begin{theorem}\cite{M0}\label{Dob}
Let $(X, X_+ , \ck ,f)$ be an OBSB and $T,S : X\to X$ be Markov
operators. The following assertions hold:
\begin{itemize}
\item[(i)] $0 \leq \d (T) \leq 1$; \item[(ii)] $|\d (T) - \d (S)|
\leq \d (T-S) \leq \left\| T - S\right\|$; \item[(iii)] $\d (TS)
\leq \d (T) \d (S)$; \item[(iv)] if $H: X\to X$ is a linear
bounded operator such that $H^* (f) =0$, then $\left\| TH \right\|
\leq \d (T) \left\| H \right\|$; \item[(v)] one has
$$
\displaystyle \d (T) = \frac{1}{2} \sup_{u,v \in \ck} \left\| Tu -
Tv\right\|;
$$
\item[(vi)] if $\d (T)=0$, then there exists $y_0 \in X_+$ such
that $T=T_{y_0}$.
\end{itemize}
\end{theorem}

\begin{remark}
Note that taking into account Theorem \ref{Dob}(v) we obtain that
the weak ergodicity (resp. weak mean ergodicity) is equivalent to
the condition $\d(T^{n})\to 0$ (resp. $\d(A_n(T))\to 0$) as
$n\to\infty$.
\end{remark}

The following theorem gives us the conditions that are equivalent
to the uniform asymptotical stability.

\begin{theorem}\cite{M0}\label{Dob1}
Let $(X, X_+ , \ck ,f)$ be an OBSB and $T : X\to X$ be a Markov
operator. The following assertions are equivalent:
\begin{itemize}
\item[(i)] $T$ is weakly ergodic; \item[(ii)] there exists $\rho
\in [0,1)$ and $n_0 \in\bn$ such that  $ \d (T^{n_0}) \leq \rho$;
\item[(iii)] $T$ is uniformly asymptotically stable.
 Moreover,
there are positive constants $C, \a, n_0 \in\bn$ and $x_0 \in\ck$
such that
$$
\left\| T^n - T_{x_0} \right\| \leq C e^{-\a n}, \,\,\,\,\,
\forall n\geq n_0.
$$
\end{itemize}
\end{theorem}

\section{Uniform mean ergodicity}

In this section, we are going to establish an analogous of Theorem
\ref{Dob1} for uniformly mean ergodic Markov operators.

Let $(X,X_+,\ck,f)$ be an OBSB. By $\gu$ we denote the set of all
Markov operators from $X$ to $X$ which have an eigenvalue 1 and
the corresponding eigenvector $f$ belongs to $\ck$.

  \begin{theorem}\label{UME}
Let  $(X, X_+, \ck, f)$ be an OBSB and $T\in \gu$. Then the
following statements are equivalent:
\begin{itemize}
\item[(i)] $T$ is weakly mean ergodic;

\item[(ii)] There exist $\rho \in [0,1)$ and $n_0 \in\bn$ such
that $\d (A_{n_0}(T)) \leq \rho$;

\item[(iii)] $T$ is uniformly mean ergodic.
\end{itemize}
\end{theorem}

\begin{proof}

The implications $(i) \Rightarrow (ii)$ and $(iii) \Rightarrow
(i)$ are obvious. It is enough to prove the implication $(ii)
\Rightarrow (iii)$.

Let us assume that there exist $n_0\in\bn$ and $\rho \in [0,1)$
such that $\d (A_{n_0}(T))\leq\rho$.

Since $T$ is Markov operator on $X$ we have
$$
\left\| A_n(T) (I - T) \right\| \leq \frac{1}{n} (1 + \left\|
T\right\|)
$$
and  for each $k\in\bn$
$$
\left\| A_n(T) (I - T^k) \right\| \leq \frac{1}{n} (1 + \left\|
T\right\| + \cdots + \left\| T^{k-1} \right\|+ \left\|
T^{n-1}\right\| + \cdots + \left\| T^{k-n+1}\right\| ).
$$
Hence, both norms converges to zero as $n\to\infty$. Therefore,
for each $m\in\bn$ one gets
\begin{eqnarray*}
\lim_{n\to\infty} \left\| A_n(T) (I- A_m(T)) \right\|& =&
\lim_{n\to\infty} \left\| A_n(T) \bigg(\frac{1}{m}
\sum_{k=0}^{m-1}
(I-T^k)\bigg) \right\| \\[2mm]
&=& \lim_{n\to\infty} \left\| \frac{1}{m} \sum_{k=0}^{m-1} A_n(T)
(I-T^k)) \right\| =0
\end{eqnarray*}
which implies
\begin{equation}\label{dd1}
\lim_{n\to\infty} \d ( A_n(T)(I- A_m(T)) =0.
\end{equation}

From (ii) Theorem \ref{Dob} one finds
\begin{equation*}
|\d ( A_n(T) A_{n_0}(T)) - \d ( A_n(T))| \leq \d ( A_n(T)(I -
A_{n_0}(T)).
\end{equation*}

From this inequality with (iii) Theorem \ref{Dob} we infer that
\begin{eqnarray}\label{dd2}
 \d ( A_n(T)(I-A_{n_0}(T))&\geq&  \d (A_n(T))-\d ( A_n(T)
 A_{n_0}(T))\nonumber\\[2mm]
&\geq&  \d (A_n(T))-\d (A_n(T))\d(A_{n_0}(T))\nonumber\\[2mm]
&\geq&  (1-\rho)\d (A_n(T))
\end{eqnarray}

So, from \eqref{dd1} and \eqref{dd2} we obtain $\lim_{n\to\infty}
\d (A_n(T))=0$, i.e.
\begin{eqnarray}\label{dd3}
\lim_{n\to\infty}\sup_{x, y\in \ck} \left\|A_n(T) x - A_n(T)y
\right\|=0.
\end{eqnarray}

Due to $T\in\gu$ one can find a fixed point $y_0\in \ck$ of $T$,
which from \eqref{dd3} yields
\begin{eqnarray*}
\lim_{n\to\infty}\sup_{x\in \ck} \left\| A_n(T)x -y_0
\right\|&\leq& \lim_{n\to\infty}\sup_{x, y\in \ck} \left\|A_n(T)x
- A_n(T)y \right\|\\[2mm]
& =&  \lim_{n\to\infty}2\d (A_n(T))=0.
\end{eqnarray*}
which means the uniform mean ergodicity of $T$. This completes the
proof.
\end{proof}

By $\gu_{ume}$ we denote the set of all uniformly mean ergodic
Markov operators belonging to $\gu$.

\begin{theorem}\label{Norm-M} Let $(X,X_+,\ck,f)$ be an OBSB. Then the set $\gu_{ume}$ is
a norm dense and open subset of $\gu$.
\end{theorem}

\begin{proof} Take an arbitrary $T\in\gu$ with a fixed point $\phi\in\ck$. Let $0<\i<2$ be an arbitrary number. Denote
$$
T^{(\i)}=\bigg(1-\frac{\i}{2}\bigg)T+\frac{\i}{2} T_\phi.
$$
It is clear that $T^{(\i)}\in\gu$, since $T^{(\i)}\phi=\phi$, and
$\|T-T^{(\i)}\|<\i$. Now we show that $T^{(\i)}\in\gu_{ume}$. It
is enough to establish that $T^{(\i)}$ is uniform asymptotically
stable (see Remark \ref{dd0}).  Indeed, by Lemma \ref{3.2}, if
$x-y\in N$, we get
\begin{eqnarray*}
\|T^{(\i)}(x-y)\|&=&\frac{\|x-y\|}{2}\|T^{(\i)}(u-v)\| \\[2mm]
&=&\frac{\|x-y\|}{2}\bigg\|\bigg(1-\frac{\i}{2}\bigg)T(u-v)+
\frac{\i}{2}T_\phi(u-v)\bigg\|\\[2mm]
&=&\frac{\|x-y\|}{2}\bigg\|\bigg(1-\frac{\i}{2}\bigg)T(u-v)\bigg\|\\[2mm]
&\leq&\bigg(1-\frac{\i}{2}\bigg)\|x-y\|
\end{eqnarray*}
which implies $\d(T^{(\i)})\leq 1-\frac{\i}{2}$. Here $u,v\in
\ck$. Hence, due to Theorem \ref{Dob1} we infer that $T^{(\i)}$ is
uniform asymptotically stable.

Now let us show that $\gu_{ume}$ is a norm open set. First, for
each $n\in\N$, we define
$$
\gu_{ume,n}=\bigg\{T\in \gu:  \ \ \d(A_n(T))< 1\bigg\}.
$$
Then one can see that
$$
\gu_{ume}=\bigcup_{n\in\N}\gu_{ume,n}.
$$
Therefore, to establish the assertion, it is enough prove that
$\gu_{ume,n}$ is a norm open set.

Take any $T\in\gu_{ume,n}$, and put $\a:=\d(A_n(T))<1$. Choose
$0<\b<1$ such that $\a+\b<1$. Let us show that
$$
\bigg\{H\in\gu: \
\|H-T\|<\frac{2\b}{n+1}\bigg\}\subset\gu_{ume,n}.
$$

We note that for each $k\in\N$ one has
\begin{eqnarray}\label{dd5}
\|H^k-T^k\|&\leq& \|H^{k-1}(H-T)\|+\|(H^{k-1}-T^{k-1})T\|\nonumber\\[2mm]
&\leq& \|H-T\|+\|H^{k-1}-T^{k-1}\|\nonumber\\[2mm]
&\cdots&\nonumber\\
&\leq& k\|H-T\|.
\end{eqnarray}

From (ii) of Theorem \ref{Dob} with \eqref{dd5} we find
\begin{eqnarray*}
|\d(A_n(H))-\d(A_n(T))|&\leq&\|A_n(H)-A_n(T)\|\\[2mm]
&\leq&\frac{1}{n}\sum_{k=0}^{n-1}\|H^{k}-T^{k}\|\\[2mm]
&\leq& \frac{1}{n}\sum_{k=1}^{n}k\|H-T\|\\[2mm]
&=& \frac{n+1}{2}\|H-T\|\\[2mm]
&<&\b
\end{eqnarray*}
Hence, the last inequality yields that
$\d(A_n(H))<\d(A_n(T))+\b<1$. This due to Theorem \ref{UME}
implies  $H\in\gu_{ume,n}$. This completes the proof.
\end{proof}

\begin{corollary} Let $T\in\gu$ be a uniformly mean ergodic Markov
operator. Then there is a neighborhood of $T$ in $\gu$ such that
every Markov operator taken from that neighborhood has a unique
fixed point.
\end{corollary}

\begin{remark} We point out that the question on the
geometric structure of the set of uniformly ergodic operators was
initiated in \cite{Hal}. The proved theorem gives some information
about the set of uniformly mean ergodic operators.
\end{remark}

\section{Perturbation Bounds and Uniform asymptotic stability of Markov operators}

In this section, we prove perturbation bounds in terms of $C$ and
$e^{\a }$ under the condition $\left\| T^n - T_{x_0} \right\| \leq
C e^{-\a n}$. Moreover, we also give several bounds in terms of
the Dobrushin's ergodicity coefficient.

\begin{theorem} \label{per1}
Let $(X, X_+, \ck, f)$ be an ordered Banach space with a base, and
$S$, $T$ be Markov operators on $X$. If $T$ is uniformly
asymptotically stable, then one has
\begin{eqnarray} \label{1}
\displaystyle
&& \left\| T^n x- S^n z \right\| \leq \\
&& \begin{cases}
\left\| x - z \right\| + n \left\| T-S\right\|, &\forall n \leq \tilde{n},\\
C e^{-\a n} \left\| x- z \right\| + ( \tilde{n}+ C \frac{e^{-\a
\tilde{n}} - e^{-\a n}}{1 - e^{-\a}}) \left\| T-S\right\|, &
\forall n > \tilde{n}
\end{cases} \nonumber
\end{eqnarray}
where $\displaystyle \tilde{n}:= \log
\bigg[\frac{\log(1/C)}{e^{-\a}}\bigg]$, $C\in \br_+$, $\a \in
\br_+$, $x,z\in \ck$.
\end{theorem}

\begin{proof}
For each $n\in \bn$,  by induction we have
\begin{equation} \label{2}
S^n = T^n + \sum_{i=0}^{n-1} T^{n-i-1} \circ (S-T) \circ S^i.
\end{equation}
Let $x,z\in \ck$ it then follows from \eqref{2} that
\begin{eqnarray*}
T^n x - S^n z &=& T^n x - T^n z - \sum_{i=0}^{n-1} T^{n-i-1} \circ (S-T) \circ S^i (z) \\
&=& T^n  (x - z) - \sum_{i=0}^{n-1} T^{n-i-1} \circ (S-T) (z_i),
\end{eqnarray*}
where $z_i=S^iz$. Hence,
\begin{equation*}
\left\| T^n x  - S^n z \right\| \leq \left\| T^n  (x - z) \right\|
+ \sum_{i=0}^{n-1} \left\|  T^{n-i-1} \circ (S-T) (z_i) \right\|.
\end{equation*}

Since $T$ and $S$ are Markov operator and due to (iv) of Theorem
\ref{Dob} one finds
\begin{equation*}
 \left\|  T^{n-i-1} \circ (S-T) (z_i) \right\| \leq \d(T^{n-i-1}) \left\| S-T\right\|
\end{equation*}
and
$$
\left\| T^n (x - z) \right\| \leq \d (T^n) \left\| x -z\right\|.
$$
Hence, we obtain
\begin{eqnarray} \label{3}
\left\| T^n x  - S^n z \right\| &\leq & \d (T^n) \left\| x - z\right\| + \sum_{i=0}^{n-1} \d (T^{n-i-1}) \left\| S-T \right\|  \nonumber\\
&=& \d (T^n)  \left\|x - z\right\| + \left\| S-T \right\|
\sum_{i=0}^{n-1} \d (T^{i}).
\end{eqnarray}

From (v) Theorem \ref{Dob} one gets
\begin{eqnarray*}
\displaystyle \d (T^i ) = \frac{1}{2} \sup_{u,v\in\ck} \left\| T^i
u- T^i v \right\| \leq \sup_{u\in \ck} \left\| T^i u- T_{x_0}
u\right\|
\end{eqnarray*}
Therefore, due to Theorem \ref{Dob1} we have
\begin{eqnarray}\label{33}
\displaystyle \d (T^n) \leq
 \begin{cases}
1, &\forall n \leq \tilde{n},\\
C e^{-\a n} , &  \forall n > \tilde{n}
\end{cases}
\end{eqnarray}
where $\tilde{n} = \bigg[\frac{\log (1/C)}{\log e^{-\a}}\bigg] =
[\log C^{\a}]$.

So, from \eqref{33}  we obtain
\begin{eqnarray} \label{4}
\displaystyle
\sum_{i=0}^{n-1} \d (T^i ) &=& \sum_{i=0}^{\tilde{n}-1} \d (T^i) + \sum_{i=\tilde{n}}^{n-1} \d (T^i) \nonumber\\
&\leq& \tilde{n} + \sum_{i=\tilde{n}}^{n-1} C e^{-\a i} \nonumber \\
&=& \tilde{n} + C e^{-a \tilde{n}} \frac{1-e^{-\a (n-
\tilde{n})}}{1- e^{-\a}}, \, \, \forall n > \tilde{n} .
\end{eqnarray}

Hence, the last inequality with \eqref{33} and \eqref{4} yields
the required assertion.
\end{proof}

\begin{corollary} \label{per2}
Let $(X, X_+, \ck, f)$ be an ordered Banach space with base and $S$,
$T$ be Markov operators on $X$. If $T$ is uniformly asymptotically
stable to $T_{x_0}$, then for every $x,y\in\ck$ one has
\begin{equation} \label{5}
\displaystyle \sup_{n\in \bn} \left\| T^n x - S^n z \right\| \leq
\left\| x - z \right\| + \bigg(\tilde{n} + C\frac{e^{-\a
\tilde{n}}}{1- e^{-\a }}\bigg) \left\| T-S\right\|.
\end{equation}
In addition, if $S$ is uniformly asymptotically stable to $S_{z_0}$,
then
\begin{equation} \label{6}
\displaystyle \left\| T_{x_0} - S_{z_0} \right\| \leq
\bigg(\tilde{n} + C\frac{e^{-\a \tilde{n}}}{1- e^{-\a }}\bigg)
\left\| T-S\right\|.
\end{equation}
\end{corollary}

\begin{proof}
The inequality \eqref{5} is a direct consequence of \eqref{1}. Now
if we consider \eqref{3} then one has
\begin{eqnarray*}
\displaystyle
\left\| T^n - S^n \right\| &=& \sup_{x, z\in \ck} \left\| T^n x - S^n z\right\| \\
&\leq& \d (T^n) \sup_{x, z \in \ck} \left\| x - z\right\| +
\sum_{i=0}^{n-1} \d (T^{n-i-1}) \left\| T-S\right\|
\end{eqnarray*}
and taking the limit as $n\to \infty$ one finds
$$
\left\| T_{x_0} - S_{z_0} \right\| \leq \left\| T-S \right\|
\sum_{i=0}^{\infty} \d (T^i).
$$
From \eqref{4} it follows that
$$
\left\| T_{x_0} - S_{z_0} \right\| \leq \left\| T-S \right\|
\bigg(\tilde{n} + C \frac{e^{-\a \tilde{n}}}{1-e^{-\a }}\bigg).
$$
This completes the proof.
\end{proof}

The inequality \eqref{3}  allows us to obtain perturbation bounds
in terms of the Dobrushin's coefficient of $T$. Namely, we have the
following result.

\begin{theorem} \label{per3}
Let $(X, X_+, \ck, f)$ be an ordered Banach space with base and $S$,
$T$ be Markov operators on $X$. If there exists a positive integer
$m$ such that $\d (T^m) < 1$ (i.e. $T$ is uniformly asymptotically
stable), then for every $x,z\in\ck$ one has
\begin{equation} \label{7}
\displaystyle \sup_{k\in \bn} \left\| T^{km} x - S^{km} z \right\|
\leq \d (T^m) \left\| x - z \right\|  + \frac{\left\| T^m - S^m
\right\|}{1 - \d (T^m)}
\end{equation}
and
\begin{eqnarray} \label{8}
\displaystyle
  \left\| T^n x - S^n z \right\| &\leq& \nonumber   \\
& & \begin{cases} \left\| x - z \right\| + \max\limits_{0<i<m}
\left\| T^i -S^i\right\|, &  n \leq
m,\\[2mm]
\d (T^m) (\left\| x - z \right\| + \max\limits_{0<i<m} \left\| T^i
-S^i\right\|) +\frac{\left\| T^m - S^m \right\|}{1 - \d (T^m)},  &
n \geq m.
\end{cases}
\end{eqnarray}

If, in addition, $S$ is uniformly asymptotically stable to
$S_{z_0}$, then
\begin{equation} \label{9}
\displaystyle \left\| T_{x_0} - S_{z_0} \right\| \leq
\frac{\left\| T^m - S^m \right\|}{1 - \d (T^m)}.
\end{equation}
\end{theorem}

\begin{proof} By the inequality
\eqref{3} for every $x,z\in\ck$ we obtain
$$
\displaystyle \sup_{n\in \bn} \left\| T^n x - S^n z \right\| \leq
\d (T^n) \left\| x - z \right\| + \left\| T-S\right\|
\frac{1}{1-\d (T)}
$$
and if $S$ is also  uniformly asymptotically stable to $S_{z_0}$
then one gets
$$
\left\| T_{x_0} - S_{z_0} \right\| \leq \frac{\left\|
T-S\right\|}{ 1- \d (T)} .
$$
If we consider $T^m$ instead of $T$, then one finds the
inequalities \eqref{7}  and \eqref{9}.

Now for every $k\in \bn$ and every integer $i $ such that $1\leq
i\leq m-1$, we have
$$
T^{mk+i} x - S^{mk+i} z = (T^{mk} - S^{mk}) T^i x + S^{mk} (T^i x
- S^i z).
$$
Therefore,
\begin{equation} \label{10}
\left\| T^{mk+i} x - S^{mk+i} z \right\| \leq \left\| T^{mk} -
S^{mk}\right\| + \d^k (S^{m}) \left\| T^i x - S^i z\right\|.
\end{equation}
Due to $T^n x - S^n z = S^n (x - z) + (T^n - S^n) x$, we get
\begin{equation} \label{11}
\left\| T^n x- S^n z \right\| \leq \left\| x -z\right\| +
 \left\| T^n - S^n \right\|
\end{equation}
where $n <m$.

If $n\geq m$ combining of \eqref{7} and \eqref{10}-\eqref{11} one
finds  \eqref{8}, which completes the proof.
\end{proof}

The following theorem gives an alternative method of obtaining
perturbation bounds in terms of $\d (T^m)$.

\begin{theorem} \label{per4}
Let $\d (T^m) < 1$ hold for some $m\in\bn$. Then for every
$x,z\in\ck$ one has

\begin{eqnarray} \label{12}
\left\| T^n x - S^n z \right\| &\leq & \d (T^m)^{\lfloor n/m
\rfloor} (\left\| x - z\right\| + \max_{0< i< m} \left\| T^i -
S^i \right\|) \\[2mm]
&&+ \frac{1 - \d (T^m)^{\lfloor n/m\rfloor}}{1 - \d (T^m)} \left\|
T^m - S^m\right\|, \,\,\,\, n \in \bn \nonumber.
\end{eqnarray}
\end{theorem}

\begin{proof} If $n < m$ then \eqref{12} reduces to \eqref{11}.
If $n\geq m $, we obtain
\begin{eqnarray*}
T^n x - S^n z &=& T^m (T^{n-m} x) - S^m (S^{n-m} z) \\
&=& T^m (T^{n-m} x - S^{n-m} z) + (T^m - S^m) S^{n-m} z.
\end{eqnarray*}
Therefore,
\begin{eqnarray*}
\left\| T^n x - S^n z \right\| \leq \left\| T^{n-m} x - S^{n-m}
z\right\| \d (T^m) + \left\| T^m - S^m \right\|.
\end{eqnarray*}

If we continue to apply this relation to
$$\left\| T^{n-m} x - S^{n-m} z\right\|, \cdots,  \left\| T^{n- m(\lfloor n/m\rfloor - 1)} x - S^{n- m(\lfloor n/m\rfloor - 1)} z \right\|
$$
 and using \eqref{11} to bound $\left\| T^{n- m \lfloor n/m\rfloor}
x- S^{n- m \lfloor n/m\rfloor} z \right\|$, we obtain
\begin{eqnarray*} \label{13}
 \left\| T^n x - S^n z \right\| &\leq& \d (T^m)^{\lfloor n/m \rfloor} (\left\| x - z\right\| +  \max_{0< i< m} \left\| T^i - S^i \right\|)
 \nonumber\\[2mm]
 && + \bigg(\d (T^m)^{\lfloor n/m \rfloor - 1} + \d
(T^m)^{\lfloor n/m \rfloor - 2} + \cdots + 1\bigg) \left\| T^m - S^m \right\| \nonumber, \\
&= & \d (T^m)^{\lfloor n/m \rfloor} (\left\| x - z\right\| +
\max_{0< i< m} \left\| T^i - S^i \right\|)\\[2mm]
&& + \frac{ 1- \d (T^m)^{\lfloor n/m \rfloor}}{1-\d (T^m)} \left\|
T^m - S^m \right\|.
\end{eqnarray*}
The proof is complete.
\end{proof}

\begin{corollary} \label{per5}
Let the condition of Theorem \ref{per4} be satisfied. Then for
every $x,z\in\ck$ we have
\begin{equation} \label{14}
\displaystyle \sup_{n\in \bn} \left\| T^n x - S^n z \right\| \leq
\sup _{n\in \bn} \d (T^m)^{\lfloor n/m\rfloor} + \frac{m \left\| T
- S\right\|}{ 1- \d (T^m)}.
\end{equation}
\end{corollary}

\begin{proof} Since $T^i - S^i = T (T^{i-1} - S^{i-1}) +
(T-S)S^{i-1}$ by induction we obtain
\begin{equation} \label{15}
\displaystyle \max_{0< i \leq m} \left\| T^i - S^i\right\| \leq m
\left\| T-S\right\| .
\end{equation}

From \eqref{12} and \eqref{15} we have
$$
 \left\| T^n x_0 - S^n z_0 \right\| \leq  \d (T^m)^{\lfloor n/m \rfloor} \left\| x_0 - z_0\right\| + m \left\| T-S\right\| \frac{1 - \d (T^m)^{\lfloor n/m\rfloor - 1}}{1 - \d (T^m)} \left\| T^m - S^m\right\|,
$$
which implies \eqref{14}.
\end{proof}

\begin{theorem} \label{per6}
If $\d (T^m) < 1$ for some $m\in\bn$, then every Markov operator
$S$ satisfying $\|S^m-T^m\|<1-\d (T^m)$ is uniformly
asymptotically stable and has a unique fixed point $z_0\in\ck$
such that
\begin{eqnarray} \label{per62}
\|x_0-z_0\|\leq \frac{\|S^m-T^m\|}{1-\d (T^m)-\|S^m-T^m\|}
\end{eqnarray}
\end{theorem}

\begin{proof} First we prove that the operator $(I-S^m)^{-1}$ is
bounded on the set $N$ (see \eqref{NN}). Indeed, take any $x\in
N$, then we have
\begin{equation}\label{per61}
\|S^mx\|\leq\|S^m-T^mx\|+\|T^mx\|\leq \rho\|x\|.
\end{equation}
where $\rho=\|S^m-T^m\|+\d(T^m)<1$. Hence by \eqref{per61} one
gets $\|S^{mn}x\|\leq \rho^n\|x\|$ for all $n\in\bn$. Therefore,
the series $\sum_{n}S^{mn}x$ converges. Using the standard
technique, one can see that
$$
(I-S^m)^{-1}x=\sum_{n}S^{mn}x
$$
and moreover, $\|(I-S^m)^{-1}x\|\leq \frac{\|x\|}{1-\rho}$, for
all $x\in N$. This means that $(I-S^m)^{-1}$ is bounded on $N$.

It is clear that the equation $S^mz_0=z_0$ with $z_0\in\ck$
equivalent to $(I-S^m)(z_0-x_0)=-(I-S^m)x_0$. Due to
$(I-S^m)x_0\in N$ we conclude the last equation has a unique
solution
$$
z_0=x_0-(I-S^m)^{-1}((I-S^m)x_0).
$$
From the identity
$$
z_0-x_0=T^m(z_0-x_0)+(S^m-T^m)(z_0-x_0)+(S^m-T^m)x_0
$$
and keeping in mind $z_0-x_0\in N$ one finds
$$
\|z_0-x_0\|\leq
\big(\d(T^m)+\|S^m-T^m\|\big)\|z_0-x_0\|+\|S^m-T^m\|
$$
which implies \eqref{per62}.

From $S^m(Sz_0)=S(S^mz_0)=Sz_0$, and the uniqueness of $z_0$ for
$S^m$ we infer that $Sz_0=z_0$. Now assume that $S$ has another
fixed point $\tilde z_0\in\ck$. Then $S^m\tilde z_0=\tilde z_0$
which yields $\tilde z_0=z_0$. Moreover, due to \eqref{per61} one
concludes that $\d(S^m)<1$, which by Theorem \ref{Dob1} yields
that $S$ is uniformly asymptotically stable. This completes the
proof.
\end{proof}

\begin{remark} The obtained results can be applied in several
directions.
\begin{itemize}
\item[(i)] We note that all obtained results extend main results
of \cite{Kar,Mit,RKW} to general Banach spaces. Hence, they allow
to apply the obtained estimates to Markov chains over various
spaces.

\item[(ii)] Considering the classical $L_p$-spaces, one may get
the perturbation bounds for uniformly asymptotically stable Markov
chains defined on these $L_p$-spaces.   On the other hand, one may
directly apply the results to Markov chains defined on more
complicated functional spaces. Moreover, by varying the Banach
space one can obtain several interesting results in the theory of
measure-valued Markov processes.

\item[(iii)] All obtained results are even new, if one takes $X$
as pre-duals of either von Neumann algebra or $JBW$-algebra.
Moreover, if we take $X$ as a dual of $C^*$-algebras, then one
gets interesting perturbation bounds for strong mixing
$C^*$-dynamical systems. If we consider non-commutative
$L_p$-spaces, then the perturbation bounds open new perspectives
in the quantum information theory (see \cite{RKW}).

\end{itemize}

\end{remark}

\section*{Acknowledgments} The first author (N.E.) thanks Hacettepe
University Scientific Research Projects Coordination Unit support
this project under the Project Number: 014 D12 601 005 -832. She
is also grateful to International Islamic University Malaysia for
kind hospitality of her research stay and this work is started
there. The second named author (F.M.) thanks also Hacettepe
University (Turkey) for kind hospitality during 5-9 September
2015, where a part of this work is carried out.
\bibliographystyle{amsplain}

\end{document}